\documentclass{amse-new}

\numberwithin{equation}{section} 

\begin{document}

 \PageNum{1}
 \Volume{2025}{Sep.}{x}{x}
 \OnlineTime{August 15, 2025}
 \DOI{0000000000000000}
 \EditorNote{Received February 24, 2025, accepted April 10, 2025\\ 1) Corresponding author}

\abovedisplayskip 6pt plus 2pt minus 2pt \belowdisplayskip 6pt
plus 2pt minus 2pt
\def\vsp{\vspace{1mm}}
\def\th#1{\vspace{1mm}\noindent{\bf #1}\quad}
\def\proof{\vspace{1mm}\noindent{\it Proof}\quad}
\def\no{\nonumber}
\newenvironment{prof}[1][Proof]{\noindent\textit{#1}\quad }
{\hfill $\Box$\vspace{0.7mm}}
\def\q{\quad} \def\qq{\qquad}
\allowdisplaybreaks[4]


\AuthorMark{Ni L. and Wallach N. }                             

\TitleMark{Compact homogeneous complex manifolds}  

\title{Compact, connected, complex manifolds that admit a compact transitive group of
holomorphic automorphisms        
}                  

\author{Lei \uppercase{Ni}$^1$}             
    {School of Mathematics, Zhejiang Normal University, Jinhua, Zhejiang, China, 321004;\\
Department of Mathematics, University of California, San Diego, La Jolla, CA
92093, USA\\
    E-mail\,$:$ leni@zjnu.edu.cn; lni.math.ucsd@gmail.com }

\author{Nolan  Wallach}     
    {Department of Mathematics, University of California, San Diego, La
Jolla, CA 92093, USA\\
    E-mail\,$:$  ucsd.nwallach@gmail.com; nwallach@ucsd.edu}

\maketitle%

\Abstract{The purpose of this paper is to develop a Lie algebraic approach to obtain new
proofs of important results of H.-C. Wang, Tits and Wolf-Wang-Ziller on
compact complex homogeneous manifolds emphasizing only those that admit a
transitive compact group of biholomorphic transformations. The method only
uses some standard results in Lie theory. The new approach provides a method
of associating a canonical abelian Lie algebra with a given integrable  complex
structure on a compact Lie algebra which extends the earlier work of Samelson
and Pittie.}      


\MRSubClass{53C30 (primary), 53C30 (secondary)}      

\section{Introduction}

In \cite{Wang}, H. C. Wang gave a classification of compact, simply connected
manifolds that admit complex structures, with a transitive action of a group
of biholomorphisms. In \cite{Tits}, Tits extended the results of Wang dropping
the simple connectedness. In Pittie \cite{Pittie}, the main results involved a
study of how $\bar{\partial}$-cohomology (i.e. the Dolbeault-cohomology) misbehaves on a compact
non-K\"{a}hler complex manifold under the deformation of the complex structures. He gave a parametrization of complex
structures on connected even dimensional compact Lie groups (whose existence
was first observed independently by Wang \cite{Wang} and by Samelson
\cite{Samelson}).

In this paper we extend the results of Wang-Tits by further development of the
methods of Pittie to compact complex manifold homogeneous under a compact
group of holomorphic transformations. The key difference in our approach to
that of Pittie for the case of even dimensional compact Lie groups, is that we
replace his topological argument with the standard Lie theoretic theorem that
the maximal compact subgroups of connected Lie group are conjugate. We
consider, the triples $(\mathfrak{g,\mathfrak{h}},J)$ with $\mathfrak{g}$ a
finite dimensional Lie algebra over $\mathbb{R}$ endowed with an invariant
inner product (that is, the Lie algebra of a compact Lie group),
$\mathfrak{\mathfrak{h}}$ a Lie subalgebra of even codimension and $J$ an
$\mathfrak{h}$--invariant complex structure (cf. Definition \ref{NDef}) on the
real vector space $\mathfrak{g}/\mathfrak{h}$ relative to which a bilinear map
$N:\mathfrak{g}/\mathfrak{h}\times\mathfrak{g}/\mathfrak{h}\rightarrow
\mathfrak{g}/\mathfrak{h}$, defined by $J$ (see also Definition \ref{NDef}),
vanishes. Here, if the corresponding Lie group $G$ is a connected Lie group
and $H$ a closed subgroup, and if $G/H$ has a $G$--invariant almost complex
structure $J$ then our $J$ and $N$ are, respectively the pull back to
$\mathfrak{g}/\mathfrak{h}$ of $J$ and its Nijenhuis tensor of the tangent
space of $G/H$ at the identity coset, which is naturally identified as
$\mathfrak{g}/\mathfrak{h}$). And $J$ is called integrable if $N=0$ (cf.
\cite{Kos}).

Set $N_{\mathfrak{g}}(\mathfrak{h)}$ equal to the normalizer of $\mathfrak{h}
$ in $\mathfrak{g}$ and if $x\in N_{\mathfrak{g}}(\mathfrak{h})$ and let
$\overline{ad}(x)$ be the action of $x$ on $\mathfrak{g}/\mathfrak{h}$ induced
by $ad(x)y=[x,y]$. With $(\mathfrak{g}, \mathfrak{h}, J)$ as above, we first have

\begin{theorem}
\label{thm:main1-wang} (i) If $J$ is an integrable complex structure on
$\mathfrak{g}/\mathfrak{h}$, let $$\mathfrak{m}=\{x\in N_{\mathfrak{g}%
}(\mathfrak{h)}\, |\, [\overline{ad}(x),J]=0\}.$$ Then $\mathfrak{m}$ is the
centralizer of an abelian subalgebra of $\mathfrak{g}$ and $[\mathfrak{m}%
,\mathfrak{m}]=[\mathfrak{h},\mathfrak{h}]$. Additionally, $\mathfrak{h}%
\subset\mathfrak{m}$.

(ii) If there exists an abelian subalgebra, $\mathfrak{t}$, of $\mathfrak{g}$
such that its centralizer, $\mathfrak{m}$, satisfies $[\mathfrak{m}%
,\mathfrak{m}]=[\mathfrak{h},\mathfrak{h}]$ then there exists a complex
structure $J$ on $\mathfrak{g}/\mathfrak{h}$ such that $N=0$.
\end{theorem}

Notice that $\mathfrak{m}\supset\mathfrak{h}$ in part (i). The above theorem
is a Lie algebraic extension of the main Theorem I and Theorem II of
\cite{Wang} (see also Theorem 4.5 of Chapter 4, Theorem 3.2 of Chapter 5 in
\cite{Mura}). In \cite{Wang}, the problem was reduced under the condition of
simple connectedness to the case that there is a compact semi-simple
transitive group of holomorphic automorphisms via a result of Montgomery
\cite{Monty} and a lemma, i.e. (2.2) of \cite{Wang}. Hence the above result
applies immediately to Wang's C-spaces. Besides the additional information provided in the
theorem, our proof is without appeal to any algebraic geometry nor do we
assume that $\mathfrak{g}$ is semisimple. Compared with the exposition of
Wang's result in \cite{Mura} our construction of $\mathfrak{m}$ in part (i) is
different from, and more direct than, that in Theorem 4.5 of Chapter 4 of
\cite{Mura}. The proof of part (ii) is more direct as well. The new ingredient
of Theorem \ref{thm:main1-wang} contains a canonical way of associating the
Lie subalgebra $\mathfrak{m}$ with $J$. The second result is related to that
of Tits.

\begin{theorem}
\label{thm:main2-tits} Let $p:\mathfrak{g}\rightarrow\mathfrak{\mathfrak{g}%
}/\mathfrak{h}$ denote natural surjection and its extension to $\mathfrak{g}%
_{\mathbb{C}}$. Let $J$ also denote the extension of $J$ to $\left(
\mathfrak{\mathfrak{g}}/\mathfrak{h}\right)  _{\mathbb{C}}%
=\mathfrak{\mathfrak{g}}_{\mathbb{C}}/\mathfrak{h}_{\mathbb{C}}$. If
$(\mathfrak{g}/\mathfrak{h)}_{+}$ is the $\mathbf{{i}}$ eigenspace for $J$ in
$\mathfrak{g}_{\mathbb{C}}/\mathfrak{h}_{\mathbb{C}}$ then the normalizer of $p^{-1}((\mathfrak{g}%
/\mathfrak{h)}_{+})$ is a parabolic subalgebra, $\mathfrak{p}$, of
$\mathfrak{g}_{\mathbb{C}}$, which as a real vector space is equal to
$\mathfrak{u}\oplus p^{-1}((\mathfrak{g}/\mathfrak{h)}_{+})$ with the
subalgebra $\mathfrak{u}$ being any subalgebras that satisfies $\mathfrak{u}%
\oplus\mathfrak{\mathfrak{h}}=\mathfrak{m}$. Here $\mathfrak{m}$ is as in
Theorem \ref{thm:main1-wang}. Also $\mathfrak{p}\cap\mathfrak{\mathfrak{g}%
}=\mathfrak{m}$.
\end{theorem}

Note that the abelian algebra  $\mathfrak{m}/\mathfrak{h}$ is canonically associated with $J$.
One of the main results in Tits \cite{Tits} can be derived from the above theorem
easily (cf. Section 4). It also leads to some classification results, which
will be described below. Let $\mathfrak{g},\mathfrak{h}$ be as above with
$\dim\mathfrak{g}/\mathfrak{h}$ even and such that there exists a subalgebra
$\mathfrak{m}$, the centralizer of an abelian subalgebra of $\mathfrak{g}$,
such that $\mathfrak{m}\supset\mathfrak{h}$ and $[\mathfrak{m},\mathfrak{m}%
]=[\mathfrak{h},\mathfrak{h}]$. If $\mathfrak{p}$ is a parabolic subalgebra of
$\mathfrak{g}_{\mathbb{C}}$ such that $\mathfrak{m}=\mathfrak{p}%
\cap\mathfrak{g}$ and $[\mathfrak{m},\mathfrak{m}]=[\mathfrak{h}%
,\mathfrak{h}]$ (this gives all choices of the Lie algebras, $\mathfrak{m}$,
appearing in Theorem \ref{Wang}) then if $J_{1}$ is a complex structure on
the even dimensional space $\mathfrak{m}/\mathfrak{h}$ (which is automatically
integrable since $\mathfrak{m}/\mathfrak{h}$ is an abelian Lie algebra) then
there exists $J(\mathfrak{p},J_{1})$ an integrable $\mathfrak{h}$--invariant
complex structure on $\mathfrak{g}/\mathfrak{h}$ such that $J_{1}%
=J_{|\mathfrak{m/\mathfrak{h}}}$. This is proved by giving a formula for
$p^{-1}\left(  \mathfrak{g}/\mathfrak{h}\right)_{+}$ equal to the nilradical
of $\mathfrak{p}$ direct sum with $p^{-1}\left(  \mathfrak{m}/\mathfrak{h}%
\right)  _{+}$.

\begin{theorem}
\label{thm:main3-moduli} Every $\mathfrak{h}$--invariant, integrable, complex
structure on $\mathfrak{g}/\mathfrak{h}$ is one of the $J(\mathfrak{p},
J_{1})$.
\end{theorem}

This generalizes the result of Pittie. In the last section of the paper we
show how these results imply those of Tits in the special case when there is a
transitive compact group of holomorphic transformations (and therefore those
of Wang on C--spaces).

We will now point out a few results that follow from our Lie algebra analysis
of complex structures and generalize the classic work of Wang and Tits. Note
that the space that we call a flag manifold (or variety) is also called a
generalized flag manifold or a D--space by Tits (the French  for flag is
drapeau). These spaces are exactly the projective varieties with a transitive
algebraic action of a reductive, affine algebraic group. If $G$ is a Lie
group, $H$ is a closed Lie subgroup of $G$, and the normalizer of $H$ in $G$,
$N_{G}(H)$, acts on $G/H$ by right translation $R_{n}gH=gn^{-1}H=gn^{-1}%
nHn^{-1}=gHn^{-1}$. If we set $L_{x}gH=xgH$ for $x\in G$ then we have the
smooth action $W_{x}=L_{x}R_{x}$ on $G/H$.

\begin{theorem}
\label{gen-Wang} Let $G$ be a connected compact Lie group and let $H$ be a
closed connected subgroup of $G$. If $G/H$ admits a $G$--invariant complex
structure, $J$, and if $M$ is the identity component of the subgroup of
$N_{G}(H)$ whose action on $T_{eH}(G/H)$ under $dW_{eH}$ is complex linear,
then $[M,M]=[H,H]$ and $G/M$ is a flag variety and $M/H$ is a complex torus
with the corresponding quotient complex structures. Thus $G/H$ admits a
holomorphic $G $--equivariant fibration
\[
M/H\hookrightarrow G/H\rightarrow G/M.
\]
That is $G/H$ is a holomorphic fiber bundle with base a flag variety and fibre
a compact complex torus with $G$--equivariant projection.
\end{theorem}

Note the above applies to any compact Hermitian homogeneous manifold (namely,
one that admits a transitive action of holomorphic isometries), hence a result which is equivalent to Theorem B of \cite{HK}. A more refined
classification with additional assumptions on the Chern connection of the
Hermitian structure was recently proved in \cite{NZ2}, without assuming the
compactness. Some special cases assuming the compactness were obtained in an
earlier paper \cite{NZ1}. Another classification was obtained for Hermitian manifolds whose Bismut connection has K\"ahler-like Bismut curvature was recently obtained in \cite{BPT} and \cite{ZZ}.

The result of Pittie, which was perhaps implicit in \cite{Samelson}, is a
special case (namely $\mathfrak{h}=\{0\}$) of Theorems \ref{thm:main1-wang}
and \ref{thm:main2-tits}, plays a crucial role in a recent work \cite{BPT} on
the study of pluriclosed manifolds with parallel Bismut torsion (see also
\cite{PZ} and reference therein for works on studies of homogeneous Hermitian
manifolds with parallel Bismut torsion). In particular, the existence of a
canonical abelian subalgebra (noting $\mathfrak{m}$ in Theorem
\ref{thm:main1-wang} is abelian when $\mathfrak{h}=\{0\}$) associated with an
integrable complex structure of a compact Lie algebra $\mathfrak{g}$
(corresponding to the special case $\mathfrak{h}=\{0\}$) holds the key
starting step of the analysis in \cite{BPT}.\footnote{The other key
observation of \cite{BPT}
is that under the assumption that the Bismut curvature $R^b$ is K\"ahler-like, the
torsion tensor $T^b$ of the Bismut connection provides a Lie algebra
structure on $T_{\mathbb{C}}M$ by $[X, Y]_2:=T^b(X, Y)$. This is related to  \cite{NZ2} where
the curvature tensor  $ R^c$ of the Chern connection provides a Lie algebra structure by $[X, Y]_1:=-R^c_{X, Y}$ on  $T_{\mathbb{C}}M\oplus \mathfrak{k}$ as the case of the Levi-Civita connection and the Riemannian curvature. Both are simple consequences of the 1st generalized Bianchi identity together with the parallelness of the  Chern torsion or the Bismut torsion.  The Jacobi identity for $[\cdot, \cdot]_i$, $i=1, 2$,  were first noted in  \cite{Nomizu} for a general setting. The Jacobi identity for $[\cdot, \cdot]_2$  was also known previously, particularly relating to the work of \cite{Agri, AFF, Ziller}. } Theorem \ref{thm:main3-moduli} leads to a classification of all
complex structures on $G/H$ in the notation of the previous theorem. Denote by
$C_{G}(H)$ the centralizer of $H$ in $G$. Set $C_{H}$ equal to the identity
component of the center of $H$.

\begin{theorem}
\label{description} Suppose that $G$ and $H$ are as in the previous theorem
and $G/H$ admits a $G $--invariant complex structure. Then all complex
structures are gotten as follows: Let $T$ be a maximal torus in $C_{G}(H)$ set
$M=TH$ then $G/M$ has a complex structure that makes it into a flag variety.
Furthermore, $M/H=T/C_{H}$ is an even dimensional torus. Each choice of a
complex structure on $G/M$ (a finite set with cardinality $\chi(G/M)$) and a
complex structure on $T/C_{H}$ yields a unique complex structure on $G/H$.
\end{theorem}
Here $\chi(G/M)$ is the Euler characteristic of $G/M$. One can refer to \cite{BH, Wallach, Wang} for the details on how to relate the total number of integrable complex structures with  $\chi(G/M)$ (as well as the order of the Weyl groups), in particular,  4.8.5, Theorem 6.2.11 and  6.6.1 of \cite{Wallach}.
\begin{theorem}
\label{Herm} Let $M$ be a connected, compact, complex manifold admitting a
homogeneous Hermitian structure. Then if $G$ is the group of all holomorphic
isometries of \thinspace$M$ and if $H$ is the stability group of a point,
$p\in M$ and if $\mathfrak{g}=Lie(G)$ and $\mathfrak{h}=Lie(H) $ and $J$ is
the pull-back of the complex structure of $M$ on $T(M)_{p}$ to $\mathfrak{g}%
/\mathfrak{h}$ then the above theorems apply to $(\mathfrak{g},\mathfrak{h}%
,J)$. In particular, $M$ is a holomorphic fiber bundle with base a flag
variety and fibre a compact complex torus with $G$--equivariant projection.
\end{theorem}

Finally we also have the following result as an application.

\begin{theorem}
\label{Irreducible} Let $(M,\left\langle \cdot, \cdot\right\rangle )$ be a
pair of a compact, connected complex manifold and a Hermitian structure such
that the group of holomorphic isometries acts transitively and isotropy
irreducibly. Then $(M,\left\langle \cdot, \cdot\right\rangle )$ is a Hermitian
symmetric space.
\end{theorem}

This was previously known for homogeneous spaces via the works \cite{Wolf, WZ}. Here we include a
self-contained direct proof of the Lie algebraic version.

The work of Wang and Tits is at least 60 years old. However, since the deepest
theorems used here are standard parts of Lie theory: Weyl's Theorem (a
connected Lie group with negative definite Killing form is compact) and
Cartan's Theorem (the maximal compact subgroups of a Lie group are conjugate)
as well as some basics on parabolic subalgebras, our methods are elementary.
Also the results on Lie algebras of this paper are strict generalizations
since, $\mathfrak{h}$ might not be the Lie algebra of any closed subgroup of
any compact Lie group with Lie algebra $\mathfrak{g}$. Our classification of
the complex structures on the spaces, Theorems \ref{gen-Wang}, \ref{description}, and  the Lie algebraic version in part (i) of Theorem \ref{thm:main1-wang}, are not
explicitly to be found in the papers of Tits \cite{Tits} and Wang \cite{Wang}. In a forth coming paper we shall study the applications of Theorems  \ref{thm:main1-wang}, \ref{thm:main2-tits}.

\section{The basic results}

If $V$ is a finite dimensional vector space over $\mathbb{R}$ then an element
$J$ in $End(V)$ such that $J^{2}=-1$ makes $V$ into a vector space over
$\mathbb{C}$ and we will call it a complex structure.

Let $J$ be a complex structure on $V$. Setting $V_{\mathbb{C}}=V\otimes
\mathbb{C}$ then we define $V_{+}$ to be the ${\bf{i}}$ eigenspace for $J$ on
$V_{\mathbb{C}}$. Then $V_{+}$ is a complex vector subspace of $V_{\mathbb{C}%
}$ such that
\[
\dim_{\mathbb{C}}V_{+}=\frac{\dim_{\mathbb{R}}V}{2}%
\]
and $V_{+}\cap V=\{0\}$. Thus as a real vector space
\[
V_{\mathbb{C}}=V\oplus V_{+}.
\]
We also observe that given a complex subspace of $V_{\mathbb{C}}$, $X$, such
that $\dim_{\mathbb{C}}X=\frac{\dim_{\mathbb{R}}V}{2}$ such that $X\cap
V=\{0\} $ there exists a unique complex structure on $V$ such that $V_{+}=X$.

If $\mathfrak{g}$ is a Lie algebra over $\mathbb{R}$ and $\mathfrak{h}$ is a
subalgebra then set
\[
N_{\mathfrak{g}}(\mathfrak{h})=\{x\in\mathfrak{g}\, |\, [x,\mathfrak{h}%
]\subset\mathfrak{h}\}.
\]
If $x\in N_{\mathfrak{g}}(\mathfrak{h})$ then $ad(x)$ induces a linear map
$\overline{ad}(x)$ on $\mathfrak{g}/\mathfrak{h}$.

\begin{lemma}
\label{lmm:11} Let $\mathfrak{g}$ be a Lie algebra over $\mathbb{R}$ and let
$\mathfrak{h}$ be a Lie subalgebra and let $J$ be a complex structure on
$\mathfrak{g}/\mathfrak{h}$ such that $[J,\overline{ad}(x)]=0$ for all
$x\in\mathfrak{h}$. If $p:\mathfrak{g}\rightarrow\mathfrak{g}/\mathfrak{h}$ is
the natural surjection and if $u,v\in\mathfrak{g}/\mathfrak{h}$ and
$x,y,x^{\prime},y^{\prime}\in\mathfrak{g}$ are such that
$p(x)=u,p(y)=v,p(x^{\prime})=Ju,p(y^{\prime})=Jv$. Then%
\[
I:=p[x,y]+J(p[x^{\prime},y]+p[x,y^{\prime}])-p[x^{\prime},y^{\prime}]
\]
is independent of the choices of $x,y,x^{\prime},y^{\prime}$.
\end{lemma}

\begin{proof}
If we add $h\in\mathfrak{h}$ to $x$ then by the formula of $I$ we get%
\[
p[x,y]+J(p[x^{\prime},y]+p[x,y^{\prime}])-p[x^{\prime},y^{\prime
}]+p[h,y]+Jp[h,y^{\prime}].
\]
Now%
\begin{eqnarray*}
p[h,y]+Jp[h,y^{\prime}]&=&\overline{ad}(h)v+J\overline{ad}(h)Jv\\
&=&0.
\end{eqnarray*}
Similarly if we replace $x^{\prime}$ by $x^{\prime}+h$ in the formula of $I$
and have%
\[
p[x,y]+J(p[x^{\prime},y]+p[x,y^{\prime}])-p[x^{\prime},y^{\prime
}]+Jp[h,y]-p[h,y^{\prime}]
\]
and similarly
\begin{eqnarray*}
Jp[h,y]-p[h,y^{\prime}]&=&J\overline{ad}(h)(v)-\overline{ad}(h)(J v)\\
&=&0.%
\end{eqnarray*}
Noting the symmetric role of $x$ and $y$, the above proves the lemma.
\end{proof}

The lemma allows us to define a bilinear map $N :\mathfrak{g}%
/\mathfrak{h\times g}/\mathfrak{h} \to \mathfrak{g}/\mathfrak{h}$ by
\[
N(u,v)=p[x,y]+J(p[x^{\prime},y]+p[x,y^{\prime}])-p[x^{\prime},y^{\prime}]
\]
for $u,v\in\mathfrak{g}/\mathfrak{h}$ and $x,y,x^{\prime},y^{\prime}$ as in
the above lemma.

\begin{definition}
\label{NDef}Let $\mathfrak{g}$ be a Lie algebra over $\mathbb{R}$ and
$\mathfrak{h}$ a Lie subalgebra. If $J\in\mathrm{End}(\mathfrak{g}%
/\mathfrak{h})$ satisfies that $J^{2}=-\operatorname{id}$ and $[J,\overline{ad}(x)]=0$ for
$x\in\mathfrak{h}$ then $J$ is called $\mathfrak{h}$--invariant (almost) complex structure;  Moreover an $\mathfrak{h}$--invariant $J$ is called
an $\mathfrak{h}$--invariant integrable complex structure if $N(u,v)=0$ for
all $u,v\in\mathfrak{g}/\mathfrak{h}$.
\end{definition}

Note that $\left(  \mathfrak{g}/\mathfrak{h}\right)  _{\mathbb{C}}$ is
naturally equal to $\mathfrak{g}_{\mathbb{C}}/\mathfrak{h}_{\mathbb{C}}$ and
we extend $p$ to a complex linear map of $\mathfrak{g}_{\mathbb{C}}$ to
$\left(  \mathfrak{g}/\mathfrak{h}\right)  _{\mathbb{C}}$.

\begin{lemma}
\label{lmm:12} Let $\mathfrak{g}$ be a Lie algebra over $\mathbb{R}$ and
$\mathfrak{h}$ a Lie subalgebra and $J$ an $\mathfrak{h}$--invariant complex
structure on $\mathfrak{g}/\mathfrak{h}$. $J$ is integrable if and only if
$p^{-1}(\left(  \mathfrak{g}/\mathfrak{h}\right)  _{+})$ is a Lie algebra.
\end{lemma}

\begin{proof}
For the necessity we need to prove that if $x,y\in p^{-1}\left(  \left(
\mathfrak{g}/\mathfrak{h}\right)  _{+}\right)  $ then $Jp[x,y]={\bf{i}}p[x,y]
$. We note that if $u\in\left(  \mathfrak{g}/\mathfrak{h}\right)  _{+}$ and
$x\in$ $\mathfrak{g}_{\mathbb{C}}$ such that $p(x)=u$ then since $Ju={\bf{i}}u$, $p({\bf{i}}x)=Ju. $ Thus, for $u=p(x),v=p(y)$ we may take $x^{\prime
}={\bf{i}}x,y^{\prime}={\bf{i}}y.$ Thus%
\begin{eqnarray*}
N(u, v)&=&p[x, y]+J(p[{\bf{i}}x, y]+p[x, {\bf{i}}y])-p[{\bf{i}}x, {\bf{i}}y]
\\
&=&2p[x, y]+2{\bf{i}}Jp[x, y]\\
&=&0.
\end{eqnarray*}
For the sufficiency we prove that if $p^{-1}(\left(  \mathfrak{g}%
/\mathfrak{h}\right)  _{+})$ is a Lie algebra then $N=0$. Note that if we read
the above argument backward then if $u,v\in\left(  \mathfrak{g}/\mathfrak{h}%
\right)  _{+} $ then $N(u,v)=0$, provided $p^{-1}(\left(  \mathfrak{g}/\mathfrak{h}%
\right)  _{+})$ is a Lie algebra. Noting that if $\left(  \mathfrak{g}%
/\mathfrak{h}\right)  _{-}$ is the $-{\bf{i}}$ eigenspace for $J$ and if
$p(x)=u\in\left(  \mathfrak{g}/\mathfrak{h}\right)  _{+}$ and $p(y)=v\in
\left(  \mathfrak{g}/\mathfrak{h}\right)  _{-}$ then taking $x^{\prime}%
={\bf{i}}x, y^{\prime}=-{\bf{i}}y$ we have
\begin{eqnarray*}
N(u,v)&=&p[x,y]+J({\bf{i}}p([x,y]-[x,y]))-p[x,y]\\
&=&0.
\end{eqnarray*}
If $u,v\in\left(  \mathfrak{g}/\mathfrak{h}\right)  _{-}$ since
$\left(  \mathfrak{g}/\mathfrak{h}\right)  _{-}$ is the complex conjugate of
$ \left(  \mathfrak{g}/\mathfrak{h}\right)  _{+}$ the above proves the same conclusion. Finally,
since \[
\left(  \mathfrak{g}/\mathfrak{h}\right)  _{\mathbb{C}}=\left(  \mathfrak{g}%
/\mathfrak{h}\right)  _{+}\oplus\left(  \mathfrak{g}/\mathfrak{h}\right)  _{-}%
\]
we complete the proof of the sufficiency, hence the claimed result.
\end{proof}

When $\mathfrak{g}$ and $\mathfrak{h}$ are the Lie algebra of corresponding
Lie group $G$, and a closed subgroup $H\subset G$, this Lemma was originally proved in  Koszul \cite{Kos}.

\begin{lemma}
\label{lmm:13} Let $J$ be an $\mathfrak{h}$--invariant, integrable complex
structure on $\mathfrak{g}/\mathfrak{h}$ then $p^{-1}((\mathfrak{g}%
/\mathfrak{h)}_{+})\cap\mathfrak{g=\mathfrak{h}}$.
\end{lemma}

\begin{proof}
Suppose that $x\in\mathfrak{g}_{\mathbb{C}}$ and $p(x)\in(\mathfrak{g}/\mathfrak{h)}_{+} $
then $p(x)\in(\mathfrak{g}/\mathfrak{h)}_{+}\cap(\mathfrak{g}/\mathfrak{h)}=0
$. Thus $x\in\mathfrak{h}$.
\end{proof}

 The above result can also be found in Theorem 3.4  of Chapter 3 of \cite{Mura}. One can easily deduce that
 \begin{eqnarray}\label{eq:Kos1}
 \mathfrak{g}_{\mathbb{C}}&=&\mathfrak{g}+p^{-1}\left(  (\mathfrak{g}/\mathfrak{h)}_{+}\right), \\ \mathfrak{g}_{\mathbb{C}}&=& p^{-1}\left((\mathfrak{g}/\mathfrak{h)}_{+}\right)+\overline{p^{-1}\left((\mathfrak{g}/\mathfrak{h)}_{+}\right)},\label{eq:Kos2}\\ \mathfrak{h}_{\mathbb{C}}&=&p^{-1}\left((\mathfrak{g}/\mathfrak{h)}_{+}\right)\cap \overline{p^{-1}\left((\mathfrak{g}/\mathfrak{h)}_{+}\right)}.\label{eq:Kos3}
 \end{eqnarray}

 We now assume that $\mathfrak{g}$ has a $\mathfrak{g}$--invariant inner
product $\left\langle \cdot, \cdot\right\rangle,$ which is equivalent to that
$\mathfrak{g}$ is a compact Lie algebra. If $\mathfrak{g}$ is a Lie algebra
then we set $\mathfrak{c}_{\mathfrak{g}}$ equal to the center of
$\mathfrak{g}$.

For a subalgebra $\mathfrak{h}$ of $\mathfrak{g}$ let
$\mathfrak{k=[\mathfrak{h}},\mathfrak{h}]$. Then $\mathfrak{k}$ is compact
semisimple and $\mathfrak{h}=\mathfrak{k}\oplus\mathfrak{c}_{\mathfrak{h}}$.

Let $G_{\mathbb{C}}^{\prime}$ denote a connected, simply connected, Lie group
with Lie algebra $[\mathfrak{g},\mathfrak{g}]$ and let $\Gamma$ be a lattice
in $\mathfrak{c}_{\mathfrak{g}}$ (that is a discrete co-compact subgroup of
$\mathfrak{c}_{\mathfrak{g}}$ under addition). Set $G_{\mathbb{C},\Gamma
}=\left(  \mathfrak{c}_{\mathfrak{g}}\right)  _{\mathbb{C}}/\Gamma\times
G_{\mathbb{C}}^{\prime}$. Weyl's Theorem implies that the connected subgroup,
$G^{\prime}$, of $G_{\mathbb{C}}^{\prime}$ corresponding to $[\mathfrak{g}%
,\mathfrak{g}]$ is compact. Set $G_{\Gamma}=\mathfrak{c}_{\mathfrak{g}}%
/\Gamma\times G^{\prime}\subset G_{\mathbb{C},\Gamma}$. Then $G_{\Gamma}$ is a
maximal compact subgroup of $G_{\mathbb{C},\Gamma}$. For the rest of this
section $\Gamma$ will be fixed and we will write $G=G_{\Gamma}$ and
$G_{\mathbb{C}}=G_{\mathbb{C},\Gamma}$.

The next result is critical to our approach.
\begin{proposition}
\label{prop:11} Let $J$ be an $\mathfrak{h}$--invariant, integrable, complex
structure on $\mathfrak{g}/\mathfrak{h}$. Then, setting $\mathfrak{l}%
=p^{-1}\left(  (\mathfrak{g}/\mathfrak{h)}_{+}\right)  $ and $L$ the connected
subgroup of $G_{\mathbb{C}}$ with Lie algebra $\mathfrak{l}$, we have $G_{\mathbb{C}}=\bar
{L}G$ (here $\bar{L}$ is the closure of $L$ in $G_{\mathbb{C}}$).
\end{proposition}

\begin{proof}
Let $n=\dim\mathfrak{g}$, $h=\dim\mathfrak{h}$. We have the exact sequence
\[
0\rightarrow\mathfrak{h}_{\mathbb{C}}\rightarrow\mathfrak{l}\rightarrow
(\mathfrak{g}/\mathfrak{h)}_{+}\rightarrow0.
\]
Thus
\begin{eqnarray*}
\dim_{\mathbb{C}}\mathfrak{l}&=&\frac{\dim_{\mathbb{R}}\left(  \mathfrak{g}%
/\mathfrak{h}\right)  }{2}+\dim_{\mathbb{R}}\mathfrak{h}\\
&=&\frac{n-h}{2}
+h\\
&=&\frac{n+h}{2}.
\end{eqnarray*}
Hence
\begin{eqnarray*}
\dim_{\mathbb{R}}\mathfrak{l}+\dim_{\mathbb{R}}\mathfrak{g}&=&2n+h\\&\geq&
\dim_{\mathbb{R}}\mathfrak{g}_{\mathbb{C}}.
\end{eqnarray*}

This implies that $\overline{L}G$ is open in $G_{\mathbb{C}}$. Indeed, the
inverse function theorem implies that $LG$ contains $U$, an open neighborhood
of the identity. Thus, since
\[
\overline{L}G=\cup_{h\in\overline{L},g\in G}hUg,
\]
it is an open subset of $G_{\mathbb{C}}$. On the other hand $G$ is compact
so $\overline{L}G$ is closed.
\end{proof}

\begin{proposition}
\label{prop:12} Let $J$ be an $\mathfrak{h}$--invariant, integrable, complex
structure on $\mathfrak{g}/\mathfrak{h}$ and let $\mathfrak{l}$ and $\bar{L}$
be as in the previous proposition. If $\mathfrak{v}$ is a real subalgebra of
$\mathfrak{l} $ such that the connected subgroup of $G_{\mathbb{C}}$ with Lie
algebra $\mathfrak{v}$ is compact then there exists $h\in\overline{L}$ such
that $\mathfrak{v}\subset Ad(h)\mathfrak{h}$.
\end{proposition}

\begin{proof}
Let $R$ denote the connected subgroup of $G_{\mathbb{C}}$ with Lie algebra
$\mathfrak{v}$. Then there exists, $U$, a maximal compact subgroup of
$G_{\mathbb{C}}$ with $R\subset U$. The conjugacy theorem for maximal compact
subgroups implies that there exists $g\in G_{\mathbb{C}}$ such that
$U=gGg^{-1}$. Proposition \ref{prop:11} implies that $g=hu$ with
$h\in\overline{L}$ and $u\in G$ thus $U=hGh^{-1}$. Now noting that
$Ad(h)\mathfrak{l}=\mathfrak{l}$, we have
\begin{eqnarray*}
\mathfrak{v}&\subset& \mathfrak{l}\cap Lie(U)\\
&=&\mathfrak{l}\cap Ad(h)\mathfrak{g}\\
&=&Ad(h)(\mathfrak{l}\cap\mathfrak{g})\\
&=&Ad(h)\mathfrak{h}.
\end{eqnarray*}
Here we have used Lemma \ref{lmm:13}, namely $\mathfrak{l}\cap\mathfrak{g}=\mathfrak{h}$.
\end{proof}

\begin{corollary}
\label{coro:11} Notation as in the previous proposition. Then $\mathfrak{l}$ has a
Levi decomposition as  $\mathfrak{l}=\mathfrak{r}\oplus\mathfrak{k}_{\mathbb{C}}$ (here
$\mathfrak{r}$ is the radical of $\mathfrak{l}$). Recall that
$\mathfrak{k}_{\mathbb{C}}$ is semisimple since $\mathfrak{k}=[\mathfrak{h}, \mathfrak{h}]$ is semisimple.
\end{corollary}

\begin{proof}
Let $\mathfrak{r}$ be the radical of $\mathfrak{l}$ and let $\mathfrak{z}$ be
a semisimple, complex subalgebra of $\mathfrak{l}$ such that $\mathfrak{l}%
=\mathfrak{r}\oplus\mathfrak{z}$ (i.e. a Levi decomposition of $\mathfrak{l}
$). We assert that \ there exists $h\in\overline{L}$ such that $\mathfrak{z}%
\subset Ad(h)\mathfrak{k}_{\mathbb{C}}$. Indeed, let $\mathfrak{z}_{c}$ be a
maximal compact subalgebra of $\mathfrak{z}$. Then since $\mathfrak{z}$ is
semisimple, the connected subgroup of $G_{\mathbb{C}}$ corresponding to
$\mathfrak{z}_{c}$ is compact hence the preceding proposition implies that
there exists $h\in\overline{L}$ such that $\mathfrak{z}_{c}\subset
Ad(h)\mathfrak{h}$. \ But since $\mathfrak{z}_{c}$ is semisimple,
$\mathfrak{z}_{c}=[\mathfrak{z}_{c},\mathfrak{z}_{c}]$. Thus, $\mathfrak{z}%
\subset Ad(h)\mathfrak{k}_{\mathbb{C}}$. But $\mathfrak{z}$ is maximal among
the semisimple subalgebras of $\mathfrak{l}$. Hence, $\mathfrak{z}%
=Ad(h)\mathfrak{k}_{\mathbb{C}}$. This implies the corollary since
$Ad(h)\mathfrak{r}=\mathfrak{r}$ and $Ad(h)\mathfrak{l}=\mathfrak{l}$.
\end{proof}

In the case that $\mathfrak{h}=\{0\}$, $\mathfrak{l}$ (=$\mathfrak{r}$) is the Samelson algebra $\mathfrak{s}$ in \cite{Pittie},
where he proves its solvability via a topological argument.  The above  lemma gives an alternate proof of this Pittie's result as a special case.

We continue with $\mathfrak{g,\mathfrak{h}}$ and $J$, an integrable
$\mathfrak{h}$--invariant complex structure on $\mathfrak{g}%
/\mathfrak{\mathfrak{h}}$, as above. Let $\mathfrak{b}_{3}$ be a Borel subalgebra of
$\mathfrak{k}_{\mathbb{C}}$ then $\mathfrak{b}_{1}=\mathfrak{b}_{3}%
\oplus\mathfrak{r}$ is a Borel subalgebra of $p^{-1}(\mathfrak{g}%
/\mathfrak{\mathfrak{h)}}_{+}$. We note that $\left(  \mathfrak{c}%
_{\mathfrak{h}}\right)  _{\mathbb{C}}\subset\mathfrak{r}$ and that
$\mathfrak{r=}\left(  \mathfrak{c}_{\mathfrak{h}}\right)  _{\mathbb{C}}%
\oplus\mathfrak{r}^{\prime}$ a direct sum of Lie algebras. Also,
$\mathfrak{b}_{2}=\left(  \mathfrak{c}_{\mathfrak{h}}\right)  _{\mathbb{C}%
}\oplus\mathfrak{b}_{3}$ is a Borel subalgebra of $\mathfrak{h}_{\mathbb{C}}$
and
\begin{equation}\label{eq:borelofl}
\mathfrak{b}_{1}\mathfrak{=\mathfrak{b}}_{2}\oplus\mathfrak{r}^{\prime}%
\end{equation}
with $\mathfrak{r}^{\prime}$ a subalgebra of $\mathfrak{r}$. Let
$\mathfrak{b}$ be a Borel subalgebra of $\mathfrak{g}_{\mathbb{C}}$ such that
$\mathfrak{b}\supset\mathfrak{b}_{1}$. Note
$$
\mathfrak{b}_3 (\subset \mathfrak{k}_{\mathbb{C}})\subset \mathfrak{b}_2(\subset \mathfrak{h}_\mathbb{C})\subset \mathfrak{b}_1 (\subset \mathfrak{l}).
$$
From (\ref{eq:Kos2}) and (\ref{eq:Kos3}) we have that
\begin{equation}\label{eq:dim1}
\dim_{\mathbb{C}}(\mathfrak{l})=\frac{\dim (\mathfrak{g})+\dim(\mathfrak{h})}{2}.
\end{equation}
Since from $\mathfrak{r=}\left(  \mathfrak{c}_{\mathfrak{h}}\right)  _{\mathbb{C}}%
\oplus\mathfrak{r}^{\prime}$  and $\mathfrak{l}=\mathfrak{r}\oplus \mathfrak{k}_\mathbb{C}$
\begin{eqnarray*}
\dim_\mathbb{C}(\mathfrak{r}^{\prime})&=&\dim_\mathbb{C}(\mathfrak{r})-
\dim_\mathbb{C}((\mathfrak{c}_\mathfrak{h})_\mathbb{C}), \mbox{ and } \\
\dim_\mathbb{C}(\mathfrak{r})&=&\dim_\mathbb{C}(\mathfrak{l})-\dim_\mathbb{C}(\mathfrak{k}_\mathbb{C})
\end{eqnarray*}
we have,  via (\ref{eq:dim1}), that
\begin{eqnarray}
\dim_\mathbb{C}(\mathfrak{r}^{\prime})&=&\dim_\mathbb{C}(\mathfrak{l})-\dim_{\mathbb{C}}(\mathfrak{h}_\mathbb{C})\nonumber \\
&=&
\frac{\dim (\mathfrak{g})-\dim(\mathfrak{h})}{2}.\label{eq:dim2}
\end{eqnarray}

If $\mathfrak{z}$ is a reductive Lie algebra over $\mathbb{C}$ and
$\mathfrak{v}$ is a Cartan subalgebra of $\mathfrak{z}$ then $\Phi
(\mathfrak{z,v)}$ will denote its root system. We continue with the situation above.

\begin{proposition}
\label{prop:13} Let $\mathfrak{b},\mathfrak{b}_{1}%
,\mathfrak{b}_{2}$ be as above and $\mathfrak{a=\mathfrak{b\cap\mathfrak{g}}}$
\ and $\mathfrak{a}_{2}=\mathfrak{b}_{2}\cap\mathfrak{h}$ then $\mathfrak{a}$
is a Cartan subalgebra of $\mathfrak{g}$, and $\mathfrak{a}_{2}$ is a Cartan
subalgebra of $\mathfrak{h}$ which  also equals to $\mathfrak{a\cap\mathfrak{h}}$.
Set $\mathfrak{u}$ equal to a complement to $\mathfrak{a\cap\mathfrak{h}}$ in
$\mathfrak{a.}$ Then $\mathfrak{a}_{\mathbb{C}}=\mathfrak{u\oplus
a}_{\mathbb{C}}\cap\mathfrak{b}_{1}$ as a real vector space. If $\alpha\in
\Phi(\mathfrak{g}_{\mathbb{C}},\mathfrak{a}_{\mathbb{C}})$ and satisfies $\alpha
_{|\mathfrak{a}_{\mathbb{C}}\cap\mathfrak{b}_{1}}=0$ then $\alpha=0$. Moreover  if
$\alpha,\beta\in\Phi(\mathfrak{g}_{\mathbb{C}},\mathfrak{a}_{\mathbb{C}})$ are
such that
\[
\alpha_{|\mathfrak{a}_{\mathbb{C}}\cap\mathfrak{b}_{1}}=\beta_{|\mathfrak{a}%
_{\mathbb{C}}\cap\mathfrak{b}_{1}}%
\]
then $\alpha=\beta$.
\end{proposition}

\begin{proof}
Note that if $\mathfrak{b}$ is a Borel subalgebra of $\mathfrak{g}%
_{\mathbb{C}}$ then $\dim_{\mathbb{C}}\mathfrak{b=}\frac{n+l_{\mathfrak{g}}%
}{2}$ where $n=\dim\mathfrak{g}$, and  if $\mathfrak{q}$ is subalgebra of
$\mathfrak{g}$ then $l_{\mathfrak{q}}$ is the (real) dimension of a Cartan subalgebra
of $\mathfrak{q}$. The Iwasawa decomposition of $G_{\mathbb{C}}$ or $\mathfrak{g}_{\mathbb{C}}$ (cf. Theorem 5.1 on p. 270 of \cite{Helgason} and Section 7.5, pp. 158--162 of \cite{Wallach} for the case that $\mathfrak{g}_{\mathbb{C}}$ is semisimple, from which the general case follows) implies that
as a real vector space%
\[
\mathfrak{g}_{\mathbb{C}}=\mathfrak{b+\mathfrak{g.}}%
\]
Thus we have the exact sequence of real vector spaces%
\[
0\rightarrow\mathfrak{b\cap\mathfrak{g\rightarrow}b\oplus
\mathfrak{g\rightarrow\mathfrak{g}}}_{\mathbb{C}}\rightarrow0.
\]
Hence $\dim_{\mathbb{R}}\mathfrak{b\cap\mathfrak{g}}=n+l_{\mathfrak{g}%
}+n-2n=l_{\mathfrak{g}}$. Since $\mathfrak{b}$ is solvable and $\mathfrak{g}$
has a positive definite inner product, $\mathfrak{a=b\cap\mathfrak{g}}$ is
abelian hence a Cartan subalgebra of $\mathfrak{g}$ by the above dimension counting. A same argument proves that $\mathfrak{b}%
_{2}\cap\mathfrak{h}$ is the Cartan subalgebra of $\mathfrak{h}$. From this we deduce that $\mathfrak{a}\cap \mathfrak{h}=\mathfrak{b}\cap \mathfrak{h}=\mathfrak{b}_2\cap \mathfrak{h}$, which implies that $$\dim_{\mathbb{R}}\mathfrak{u}= l_\mathfrak{g}-l_\mathfrak{h}.$$
 We also note that%
\[
\dim_{\mathbb{C}}\mathfrak{b}=\dim_{\mathbb{C}}\mathfrak{b}_{1}+\frac
{l_{\mathfrak{g}}-l_{\mathfrak{h}}}{2}.
\]
Indeed, by (\ref{eq:borelofl}) and (\ref{eq:dim2}) we have
\begin{eqnarray*}
\dim_{\mathbb{C}}\mathfrak{b}_{1}&=&\frac{n-\dim\mathfrak{h}}{2}+\frac
{\dim\mathfrak{h}+l_{\mathfrak{h}}}{2}
\\
&=&\frac{n+l_{\mathfrak{g}}}{2}-\frac{l_{\mathfrak{g}}-l_{\mathfrak{h}}}{2}.
\end{eqnarray*}

We now show that
\[
\mathfrak{b=\mathfrak{u\oplus\mathfrak{b}}}_{1}%
\]
as real vector spaces. Indeed, by Lemma \ref{lmm:13}, $p^{-1}((\mathfrak{g}/\mathfrak{h})_{+})%
\cap\mathfrak{g=\mathfrak{h}}$. Thus $\mathfrak{b}_{1}\cap\mathfrak{g\subset
\mathfrak{h\cap\mathfrak{a}}}$, which implies  $\mathfrak{b}_{1}\cap\mathfrak{u}=0$. But
$\mathfrak{b}_{1}+\mathfrak{u\subset\mathfrak{b}}$. And by the above%
\begin{eqnarray*}
\dim_{\mathbb{R}}\left( \mathfrak{b}_{1}+\mathfrak{u}\right)  &=&2\dim
_{\mathbb{C}}\mathfrak{b}_{1}+l_{\mathfrak{g}}-l_{\mathfrak{h}}\\
&=&n+l_{\mathfrak{h}}\\
&=&\dim_{\mathbb{R}}\mathfrak{b}.
\end{eqnarray*}
Putting them together proves  the assertion $\mathfrak{b}=\mathfrak{u}\oplus\mathfrak{b}_{1}$ as vector spaces.

Since $\mathfrak{u\subset\mathfrak{a}}_{\mathbb{C}}$, we have the following decomposition of real
vector spaces
\begin{eqnarray}
\mathfrak{a}_{\mathbb{C}}&=&\mathfrak{a}_{\mathbb{C}}\cap\mathfrak{u}
\oplus\mathfrak{a}_{\mathbb{C}}\cap\mathfrak{b}_{1}\nonumber\\
&=&\mathfrak{u}\oplus\mathfrak{a}_{\mathbb{C}}\cap\mathfrak{b}_{1}.\label{eq:a}
\end{eqnarray}
In fact, for any $x\in \mathfrak{a}_{\mathbb{C}}\subset \mathfrak{b}$, write it as $x=y+z$ with $y\in \mathfrak{u}$ and $z\in \mathfrak{b}_1$. Hence $z=x-y\in \mathfrak{a}_{\mathbb{C}}$ since $y\in \mathfrak{u}\subset \mathfrak{a}_{\mathbb{C}}$.

If $\alpha\in\Phi(\mathfrak{g}_{\mathbb{C}},\mathfrak{a}_{\mathbb{C}})$
(respectively $\gamma,\delta\in\Phi(\mathfrak{g}_{\mathbb{C}},\mathfrak{a}%
_{\mathbb{C}})$) and $\alpha_{|\mathfrak{b}_{1}\cap\mathfrak{a}_{\mathbb{C}}%
}=0 $ (respectively $\left(  \gamma-\delta\right)  _{|\mathfrak{b}_{1}%
\cap\mathfrak{a}_{\mathbb{C}}}=0$) then $\alpha=0$ (resp. $\gamma=\delta$).
Indeed, if $\lambda=\alpha$ or $\lambda=\gamma-\delta$ then $\lambda
(\mathfrak{a}_{\mathbb{C}})=\lambda(\mathfrak{u)}\subset{\bf{i}}\mathbb{R}$
since $\Phi(\mathfrak{g}_{\mathbb{C}},\mathfrak{a}_{\mathbb{C}}%
)_{|\mathfrak{a}}\subset{\bf{i}}\mathfrak{a}^{\ast}$. But $\lambda$ is
complex linear, so $\lambda=0$.
\end{proof}

If $\alpha\in\Phi(\mathfrak{g}_{\mathbb{C}},\mathfrak{a}_{\mathbb{C}})$ then
set $\mathfrak{g}_{\alpha}=\{X\in\mathfrak{g}_{\mathbb{C}}\, |\,
[H,X]=\alpha(H)X, H\in\mathfrak{a}_{\mathbb{C}}\}$. Then Proposition \ref{prop:13}
implies%
\[
\mathfrak{g}_{\alpha}=\{X\in\mathfrak{g}_{\mathbb{C}}\, |\, [H, X]=\alpha
(H)X,H\in\mathfrak{a}_{\mathbb{C}}\cap\mathfrak{b}_{1}\}.
\]
Let $\Phi^{+}=\{\alpha\in\Phi(\mathfrak{g}_{\mathbb{C}},\mathfrak{a}%
_{\mathbb{C}})\,|\, \mathfrak{g}_{\alpha}\subset\mathfrak{b}\}.$ Since the weight
spaces of $\mathfrak{a}_{\mathbb{C}}\cap\mathfrak{b}_{1}$ are the same as
those of $\mathfrak{a}_{\mathbb{C}}$ and $\mathfrak{a}_{\mathbb{C}%
}+\mathfrak{b}_{1}=\mathfrak{b}$ we have%
\[
\mathfrak{b}_{1}=\mathfrak{a}_{\mathbb{C}}\cap\mathfrak{b}_{1}\oplus
\bigoplus_{\alpha\in\Phi^{+}}\mathfrak{g}_{\alpha}%
\]
and as real Lie algebra
\begin{equation}\label{eq:main1}
\mathfrak{b}=\mathfrak{\mathfrak{u}}\oplus\mathfrak{b}_{1},
\end{equation}
with  $\mathfrak{u}$ and $\mathfrak{b}_1$ being two ideals of $\mathfrak{b}$.

Note that $\mathfrak{b}$ is uniquely determined by $\mathfrak{b}_1$ (which  is given by $\mathfrak{b}_3\oplus \mathfrak{r}$) since $[\mathfrak{b}, \mathfrak{b}]=[\mathfrak{b}_1, \mathfrak{b}_1]$,  and  since a Borel subalgebra is the normalizer of its commutator ideal.  In the case $\mathfrak{h}=\{0\}$ it implies that $J$ uniquely determines a Cartan subalgebra $\mathfrak{a}\subset \mathfrak{g}$ since $\mathfrak{b}$ is uniquely determined by $[\mathfrak{l}, \mathfrak{l}]$, with $\mathfrak{l}=\mathfrak{g}_+$ ($=\mathfrak{b}_1$).

\begin{lemma}
\label{lmm:14}
\[
\lbrack\mathfrak{a}_{\mathbb{C}},p^{-1} ((\mathfrak{g}/\mathfrak{h})_+)]\subset
p^{-1} ((\mathfrak{g}/\mathfrak{h})_+)%
\]
and
\[
\lbrack\mathfrak{a,\mathfrak{h}}]\subset\mathfrak{h.}%
\]
\end{lemma}

\begin{proof}
Clearly $[\mathfrak{a}_{\mathbb{C}}\cap\mathfrak{b}_{1},p^{-1}((\mathfrak{g}/\mathfrak{h})_{+})]\subset p^{-1} ((\mathfrak{g}/\mathfrak{h})_+).$ This
implies that there exists $\Lambda\subset\Phi(\mathfrak{g}_{\mathbb{C}%
},\mathfrak{a}_{\mathbb{C}})$ such that%
\[
p^{-1} ((\mathfrak{g}/\mathfrak{h})_+)=\mathfrak{a}_{\mathbb{C}}\cap
p^{-1} ((\mathfrak{g}/\mathfrak{h})_+)\oplus\bigoplus_{\alpha\in\Lambda
}\mathfrak{g}_{\alpha}.
\]
Since all of these spaces are $\mathfrak{a}_{\mathbb{C}}$ invariant the first
assertion follows. For the second,%
\[
\lbrack\mathfrak{a},\mathfrak{h}]=[\mathfrak{a},p^{-1}((\mathfrak{g}/\mathfrak{h})_{+})\cap\mathfrak{g}]\subset p^{-1}((\mathfrak{g}/\mathfrak{h})_{+})\cap\mathfrak{g}=\mathfrak{h},%
\]
since $[\mathfrak{a}, p^{-1} ((\mathfrak{g}/\mathfrak{h})_+)]\subset p^{-1} ((\mathfrak{g}/\mathfrak{h})_+)$ and $[\mathfrak{a}, \mathfrak{g}]\subset \mathfrak{g}$.
\end{proof}

For our next results we will be using some well known material about parabolic Lie subalgebras of
$\mathfrak{g}_{\mathbb{C}},$ which we now recall. By definition these are
subalgebras, $\mathfrak{p}$, that contain a Borel  subalgebra of $\mathfrak{g}_\mathbb{C}$. If $\mathfrak{p}$ is
a parabolic subalgebra of $\mathfrak{g}_{\mathbb{C}}$ then its nil-radical  is
an ideal maximal among ideals consisting of nilpotent elements.  Recall that
$x\in\mathfrak{g}_{\mathbb{C}}$ is nilpotent if $x\in\lbrack\mathfrak{g}%
_{\mathbb{C}},\mathfrak{g}_{\mathbb{C}}]$ and $ad(x)$ is nilpotent. The
nil-radical is unique by definition.  We will use the notation $\tau$ for complex conjugation
on $\mathfrak{g}_{\mathbb{C}}$ with respect to $\mathfrak{g}$. If
$\mathfrak{p}$ is a parabolic subalgebra of $\mathfrak{g}_{\mathbb{C}}$ then
$\mathfrak{p}\cap\tau\left(  \mathfrak{p}\right)$
is a reductive subalgebra that is the centralizer in $\mathfrak{g}_{\mathbb{C}%
}$ of its center (cf. \cite{Wallach2}, Section 2.2.6, pp. 50--51).
And if $\mathfrak{n}$ is the nil-radical of $\mathfrak{p}$ then the Langlands decomposition of $\mathfrak{p}$
implies that
\[
\mathfrak{p}=\mathfrak{p}\cap\tau\left(  \mathfrak{p}\right)  \oplus\mathfrak{n}
\]
and that $\mathfrak{p}$ is the normalizer in $\mathfrak{g}_{\mathbb{C}}$ of $\mathfrak{n}$ (cf. \cite{Wallach2}, p. 52). Finally, if $\mathfrak{m}=\mathfrak{p\cap\mathfrak{g}}$ then $\mathfrak{p}%
\cap\tau\left(  \mathfrak{p}\right)  =\mathfrak{m}_{\mathbb{C}}$. It is also known that $\mathfrak{m}$ is the centralizer (in $\mathfrak{g}$) of its center.

We now return to the situation/notaions of the preceding Lemma \ref{lmm:14}.

\begin{theorem}
\label{thm:The-parabolic}$\mathfrak{p}=\mathfrak{a}_{\mathbb{C}}%
+p^{-1}((\mathfrak{g}/\mathfrak{h})_{+})$ is a parabolic subalgebra of
$\mathfrak{g}_{\mathbb{C}}$. Set $\mathfrak{m}=\mathfrak{p}\cap\mathfrak{g}$
then $\mathfrak{m}\supset\mathfrak{h}$ and $[\mathfrak{m},\mathfrak{m}%
]=[\mathfrak{h},\mathfrak{h}]$. The normalizer of $p^{-1} ((\mathfrak{g}/\mathfrak{h})_+)$ in $\mathfrak{g}_{\mathbb{C}}$ is $\mathfrak{p}$. Moreover, $\mathfrak{p}=\mathfrak{u}\oplus p^{-1} ((\mathfrak{g}/\mathfrak{h})_+)$.
\end{theorem}

\begin{proof} By Lemma \ref{lmm:14}, $\mathfrak{p}$ as defined in the statement of the theorem  is a subalgebra. Moreover the Borel subalgebra  $\mathfrak{b}$ of $\mathfrak{g}_{\mathbb{C}}$ is contained in $\mathfrak{p}$ by (\ref{eq:main1}) so $\mathfrak{p}$ is a parabolic
subalgebra. Since $\mathfrak{b}=\mathfrak{u}\oplus \mathfrak{b}_1$, $\mathfrak{b}_1=\mathfrak{b}\cap
p^{-1}((\mathfrak{g}/\mathfrak{h})_{+})$, and $\mathfrak{p}=\mathfrak{b}+p^{-1} ((\mathfrak{g}/\mathfrak{h})_+)$, $\mathfrak{p}=\mathfrak{u}\oplus p^{-1} ((\mathfrak{g}/\mathfrak{h})_+)$. Hence arguing as in (\ref{eq:a})
\begin{equation}\label{eq:main2}
\mathfrak{m}=\mathfrak{p}\cap\mathfrak{g}=\mathfrak{u}\oplus
p^{-1} ((\mathfrak{g}/\mathfrak{h})_+)\cap
\mathfrak{g}=\mathfrak{u}\oplus\mathfrak{h}.
\end{equation}
The preceding lemma now implies that $[\mathfrak{m},\mathfrak{m}%
]=[\mathfrak{h},\mathfrak{h}]$. Lemma \ref{lmm:14} then implies that
$p^{-1} ((\mathfrak{g}/\mathfrak{h})_+)$ is normal in $\mathfrak{p}$.  Since
$\mathfrak{p}$ is the normalizer
of its nil-radical $\mathfrak{n}$, which is contained in $p^{-1} ((\mathfrak{g}/\mathfrak{h})_+)$, $\mathfrak{p}$ is the normalizer of $p^{-1}((\mathfrak{g}/\mathfrak{h})_+)$. In fact, if $x\in [\mathfrak{g}_\mathbb{C}, \mathfrak{g}_\mathbb{C}]$ is such that $ad(x) (p^{-1} ((\mathfrak{g}/\mathfrak{h})_+))\subset p^{-1} ((\mathfrak{g}/\mathfrak{h})_+)$ then  $ad(x) (\mathfrak{n})\subset \mathfrak{n}$, since $\mathfrak{n}$ is also the nil-radical of $p^{-1} ((\mathfrak{g}/\mathfrak{h})_+)$ and the automorphism of a Lie algebra preserves its nil-radical.\end{proof}
Note that the above theorem, together with (\ref{eq:main2}), implies Theorem \ref{thm:main2-tits} in the introduction with $\mathfrak{u}$ as in Proposition \ref{prop:13}.
\begin{corollary}
\label{coro:12} Let $\mathfrak{m}$ be as in the previous theorem. Then
\[
\mathfrak{m}=\{x\in\mathfrak{g}\, |\, [x,\mathfrak{h}]\subset\mathfrak{h}%
,[\overline{ad}(x),J]=0\}.
\]
\end{corollary}

\begin{proof}
Assume that $x$ is in the right hand side of the assertion. Then
$\overline{ad}(x)(\mathfrak{g}/\mathfrak{h)}_{+}\subset(\mathfrak{g}%
/\mathfrak{h)}_{+}$, hence $ad(x)p^{-1} ((\mathfrak{g}/\mathfrak{h})_+)\subset
p^{-1} ((\mathfrak{g}/\mathfrak{h})_+)$. So the theorem above implies that
$x\in \mathfrak{p}$, thus $x\in\mathfrak{p}\cap\mathfrak{g},$  which is $\mathfrak{m}$ by the definition. On the other hand, Lemma \ref{lmm:14} and (\ref{eq:main2}) imply that the left hand side of the equation is a subset of the right hand side.
\end{proof}

\begin{corollary}
\label{BasicObs} Let the notation be as in the previous theorem. We look upon
$\mathfrak{m}/\mathfrak{h}$ as a subspace of $\mathfrak{g}/\mathfrak{h}$. Then $J$ satisfies
\[
J(\mathfrak{m}/\mathfrak{h})\subset\, \mathfrak{m}/\mathfrak{h}
\]
and it defines an integrable, $\mathfrak{m}$--invariant complex structure on
$\mathfrak{m}/\mathfrak{h}.$ If $\mathfrak{n}$ is the nil-radical of
$\mathfrak{p}$ then
\[
p^{-1} ((\mathfrak{g}/\mathfrak{h})_+)=p^{-1}((\mathfrak{m}/\mathfrak{h}%
)_{+})\oplus\mathfrak{n}\text{.}%
\]
\end{corollary}

\begin{proof}
Lemma \ref{lmm:14} implies that $[\mathfrak{m}, p^{-1}((\mathfrak{g}/\mathfrak{h})_{+})]\subset p^{-1} ((\mathfrak{g}/\mathfrak{h})_+)$. This
implies that if $x\in\mathfrak{m}$ then $[J,\overline{ad}(x)]=0$. Since
$\mathfrak{m}=\mathfrak{c}_{\mathfrak{m}}\oplus\lbrack\mathfrak{m}%
,\mathfrak{m}]$, $\mathfrak{h}=\mathfrak{c}_{\mathfrak{h}}\oplus\lbrack\mathfrak{h}%
,\mathfrak{h}]$, $[\mathfrak{m}, \mathfrak{m}]=[\mathfrak{h}, \mathfrak{h}]$,  $\mathfrak{m}/\mathfrak{h}=\mathfrak{c}_{\mathfrak{m}%
}/\mathfrak{c}_{\mathfrak{h}}=\mathfrak{a}/\left(  \mathfrak{h}\cap
\mathfrak{a}\right)  $ thus
\[
\mathfrak{m/\mathfrak{h}}=\{v\in\mathfrak{g/\mathfrak{h}}\, |\, \overline
{ad}(x)v=0,x\in\mathfrak{m}\}.
\]
This together with $[J, \overline{ad}(x)]=0$ implies
\[
J(\mathfrak{m}/\mathfrak{h})\subset\, \mathfrak{m}/\mathfrak{h}.
\]
We also note that since $\mathfrak{p=m}_{\mathbb{C}}\oplus\mathfrak{n}$ and
$\mathfrak{n}$ is contained in $p^{-1} ((\mathfrak{g}/\mathfrak{h})_+)$ the
second assertion follows.
\end{proof}

\section{The classification}

The aim of this section is to give a classification of the triples
$(\mathfrak{g},\mathfrak{h},J)$ with $\mathfrak{g}$ a finite dimensional Lie
algebra over $\mathbb{R}$ that admits an $\mathfrak{g}$-invariant inner
product $\langle \cdot, \cdot\rangle$ (i.e. is the Lie algebra of a compact Lie group), which will be fixed for the discussion below, $\mathfrak{h}$ is a
Lie subalgebra and $J$ is an integrable complex structure on $\mathfrak{g}%
/\mathfrak{h}$.

The following  result, together with Theorem \ref{thm:The-parabolic}, Corollary \ref{coro:12}, implies  Theorem \ref{thm:main1-wang} in the introduction, which contains the main results of Wang \cite{Wang}.

\begin{theorem}
\label{Wang} Let $\mathfrak{g}$ be a finite dimensional Lie algebra over
$\mathbb{R}$ that admits a $\mathfrak{g}$--invariant inner product and let
$\mathfrak{h}$ be a Lie subalgebra such that $\dim\mathfrak{g}/\mathfrak{h}$
is even. A necessary and sufficient condition that $\mathfrak{g}/\mathfrak{h}
$ admits an integrable $\mathfrak{h}$--invariant complex structure is that
there exists an abelian subalgebra, $\mathfrak{t}$, of $\mathfrak{g}$ such
that  $\mathfrak{m}=C_\mathfrak{g}(\mathfrak{t})$, the centralizer of $\mathfrak{t}$ in $\mathfrak{g}$
satisfies  $[\mathfrak{m}, \mathfrak{m]=}%
[\mathfrak{h}, \mathfrak{h]}$.
\end{theorem}

\begin{proof}
The necessity is an immediate consequence of Theorem \ref{thm:The-parabolic}.
Below we will prove the sufficiency by producing a method of finding
integrable $\mathfrak{h}$--invariant complex structures. This needs some preparation which will be carried out below.
\end{proof}

\begin{lemma}
\label{lmm:21} Let $\mathfrak{m}$ be the centralizer of an abelian subalgebra
of $\mathfrak{g}$. Then there exists a parabolic subalgebra $\mathfrak{p}$ of $\mathfrak{g}%
_{\mathbb{C}}$ such that $\mathfrak{m}=\mathfrak{p}\cap\mathfrak{g}$.
\end{lemma}

\begin{proof}
Let $\mathfrak{t}$ be an abelian subalgebra of $\mathfrak{g}$ such that
$\mathfrak{m}$ is the centralizer of $\mathfrak{t}$. If $\mathfrak{t}$ is central take $\mathfrak{p}$ =$\mathfrak{g}_{\mathbb{C}}$.
Otherwise we argue as follows:
 Let $\mathfrak{a}$ be a
maximal abelian subalgebra in $\mathfrak{g}$ containing $\mathfrak{t}$. Clearly,
$\mathfrak{a}\subset\mathfrak{m}$. By hypothesis
\[
\mathfrak{m}_{\mathbb{C}}=\mathfrak{a}_{\mathbb{C}}\oplus\bigoplus_{%
\begin{array}
[c]{c}%
\alpha\in\Phi(\mathfrak{g}_{\mathbb{C}},\mathfrak{a}_{\mathbb{C}}),\\
\alpha(\mathfrak{t})=0
\end{array}
}\mathfrak{g}_{\alpha}.
\]
Let $Q=\{\alpha\in\Phi(\mathfrak{g}_{\mathbb{C}},\mathfrak{a}_{\mathbb{C}})\,
|\, \alpha(\mathfrak{t})\neq0\}$. Since $Q$ is finite there exists $h\in
{\bf{i}}\mathfrak{t}$ such that $\alpha(h_{{}})\neq0, \forall \alpha\in Q$. Let $Q^{+}%
=\{\alpha\in Q\, |\, \alpha(h)>0\}.$ Set $\mathfrak{n=\bigoplus}_{\alpha\in
Q^{+}}\mathfrak{g}_{\alpha}$. Let $\mathfrak{b}_{1}$ be a Borel subalgebra in
$\mathfrak{m}_{\mathbb{C}}$. Then $\mathfrak{b}=\mathfrak{b}_{1}%
\oplus\mathfrak{n}$ is a Borel subalgebra of $\mathfrak{g}_{\mathbb{C}}$
contained in $\mathfrak{m}_{\mathbb{C}}\oplus\mathfrak{n}$. Thus
$\mathfrak{p}=\mathfrak{m}_{\mathbb{C}}\oplus\mathfrak{n}$ is a parabolic
subalgebra such that $\mathfrak{p}\cap\mathfrak{g=\mathfrak{m}}$.
\end{proof}
\begin{lemma}\label{lemma:HK} For any subalgebra $\mathfrak{t}\subset \mathfrak{g}$, let $C_{\mathfrak{g}}(\mathfrak{t})$ denote the centralizer of $\mathfrak{t}$.
Let $\mathfrak{h}$ be a Lie subalgebra of $\mathfrak{g}$ such that there
exists an abelian subalgebra, $\mathfrak{t}$, of $\mathfrak{g}$ such that $%
\left[ C_{\mathfrak{g}}(\mathfrak{t}),C_{\mathfrak{g}}(\mathfrak{t})\right] =%
\mathfrak{[\mathfrak{h}},\mathfrak{h}]$ then there exists an abelian subalgebra $\mathfrak{t}%
^{\prime }$  such that $\mathfrak{h}\subset C_{%
\mathfrak{g}}(\mathfrak{t}^{\prime })$ and $\left[ C_{\mathfrak{g}}(%
\mathfrak{t}^{\prime }),C_{\mathfrak{g}}(\mathfrak{t}^{\prime })\right] =%
\mathfrak{[\mathfrak{h}},\mathfrak{h}]$.
\end{lemma}

\begin{proof}
Let $\mathfrak{c}$ be the center of $C_{\mathfrak{g}}(\mathfrak{t})$ then we
assert that $\mathfrak{c}$ is a Cartan subalgebra of $C_{\mathfrak{g}}([%
\mathfrak{h},\mathfrak{h}]).$ Indeed, if $X\in C_{\mathfrak{g}}([\mathfrak{h}%
,\mathfrak{h}])$ and $[X,\mathfrak{c}]=0$ then since, $C_{\mathfrak{g}}(%
\mathfrak{t})=\mathfrak{c}\oplus \mathfrak{[\mathfrak{h}},\mathfrak{h}],$ $%
[X,C_{\mathfrak{g}}(\mathfrak{t})]=0,$ so $X\in \mathfrak{c}$. We also note
that if $\mathfrak{a}$ is a Cartan subalgebra of $C_{\mathfrak{g}}([%
\mathfrak{h},\mathfrak{h}])$ then $\mathfrak{a}\oplus \lbrack \mathfrak{h},%
\mathfrak{h}]=C_{\mathfrak{g}}(\mathfrak{a})$. To see this let $G$ be a
connected compact Lie group with Lie algebra $\mathfrak{g}$ and let $U$ be
the connected subgroup of $G$ with Lie algebra $C_{\mathfrak{g}}([\mathfrak{h%
},\mathfrak{h}])$. Then there exists $u\in U\subset G$ such that $Ad(u)%
\mathfrak{a=\mathfrak{c}}$. Thus $Ad(u)C_{\mathfrak{g}}(\mathfrak{a})=C_{%
\mathfrak{g}}(\mathfrak{c})$. Let $\mathfrak{c}^{\prime }$ be the center of $%
\mathfrak{h}$ and let $\mathfrak{w}$ be a Cartan subalgebra of $C_{\mathfrak{%
g}}([\mathfrak{h},\mathfrak{h}])$ containing $\mathfrak{c}^{\prime }$. Then $%
C_{\mathfrak{g}}(\mathfrak{w})\supset \mathfrak{c}^{\prime }\oplus \lbrack
\mathfrak{h},\mathfrak{h}]=\mathfrak{h}$ and $[C_{\mathfrak{g}}(\mathfrak{w}%
),C_{\mathfrak{g}}(\mathfrak{w})]=[\mathfrak{h},\mathfrak{h}]$.
\end{proof}
There is a  geometric proof of the Lie group analogue of this lemma in \cite{HK}.

We now give the promised construction. Let $\mathfrak{h}$ be a subalgebra of
$\mathfrak{g}$ such that there exists $\mathfrak{m}$, the centralizer of an
abelian subalgebra of $\mathfrak{g}$, such that $[\mathfrak{m}, \mathfrak{m]=}%
[\mathfrak{h}, \mathfrak{h]}$. The previous lemma allows us to assume that
$\frak{m}\supset \frak h$.  Let $P(\mathfrak{h)}$ be the set of parabolic
subalgebras, $\mathfrak{p}$, of $\mathfrak{g}_{\mathbb{C}}$ such that if
$\mathfrak{m}_{\mathfrak{p}}=\mathfrak{p\cap\mathfrak{g}}$ then  $\mathfrak{m}_{\mathfrak{p}} \supset \mathfrak{h} $ and $[\mathfrak{m}%
_{\mathfrak{p}},\mathfrak{m}_{\mathfrak{p}}]=[\mathfrak{h},\mathfrak{h}]$.  The above two lemmata imply that $P(\mathfrak{h)}$ is not empty.
Let $\mathfrak{p}\in P(\mathfrak{h})$ and  denote its nil-radical by $\mathfrak{n}_{\mathfrak{p}}$.
If $J_{1}$ is an $\mathfrak{h}$--invariant complex structure on
$\mathfrak{m}_{\mathfrak{p}}/\mathfrak{h}$ then the structure is integrable
since $\mathfrak{m}_{\mathfrak{p}}/\mathfrak{h}$ is an abelian Lie algebra.
The sufficiency in the above theorem now follows from the result below.

\begin{proposition}
\label{prop:21} There exists an integrable $\mathfrak{h}$--invariant complex
structure, $J=J(\mathfrak{p},J_{1})$ on $\mathfrak{g}/\mathfrak{h}$ such that
the ${\bf{i}}$ eigenspace for $J$ is
\[
p\left(p^{-1}\left(  (\mathfrak{m}_{\mathfrak{p}}/\mathfrak{h})  _{+}\right)%
\oplus\mathfrak{n}_{\mathfrak{p}}\right).
\]

\end{proposition}

\begin{proof}
Since $p^{-1}\left( ( \mathfrak{m}_{\mathfrak{p}}/\mathfrak{h})
_{+}\right)\supset\mathfrak{h}_{\mathbb{C}}$ we see that $$p^{-1}(p(p^{-1}\left((
\mathfrak{m}_{\mathfrak{p}}/\mathfrak{h})  _{+}\right)\oplus\mathfrak{n}%
_{\mathfrak{p}})=p^{-1}\left( ( \mathfrak{m}_{\mathfrak{p}}/\mathfrak{h}%
)  _{+}\right)\oplus\mathfrak{n}_{\mathfrak{p}}.$$ Hence%
\[
p\left(p^{-1}\left(  (\mathfrak{m}_{\mathfrak{p}}/\mathfrak{h})  _{+}\right)%
\oplus\mathfrak{n}_{\mathfrak{p}}\right)=\left(  \mathfrak{m}_{\mathfrak{p}%
}/\mathfrak{h}\right)  _{+}\oplus p(\mathfrak{n}_{\mathfrak{p}})
\]
and $p$ is injective on $\mathfrak{n}_{\mathfrak{p}}$. Now,
\begin{eqnarray*}
\dim_{\mathbb{C}}\mathfrak{n}_{\mathfrak{p}}&=&\frac{n-\dim\mathfrak{m}%
_{\mathfrak{p}}}{2},\\
  \dim_{\mathbb{C}}\left(  \mathfrak{m}_{\mathfrak{p}}/\mathfrak{h}\right)_{+}&=&\frac{\dim\mathfrak{m}_{\mathfrak{p}}%
-\dim\mathfrak{h}}{2}.
\end{eqnarray*}
The first equation above follows from expressing $\mathfrak{p}$ in terms of the root spaces and the corresponding expression for $\mathfrak{n}_{\mathfrak{p}}$.
Thus
\[
\dim_{\mathbb{C}}p\left(p^{-1}\left( ( \mathfrak{m}_{\mathfrak{p}}/\mathfrak{h}%
)  _{+}\right)\oplus\mathfrak{n}_{\mathfrak{p}}\right)=\frac{n-\dim\mathfrak{h}}{2}.
\]
We also note that %
\[
p\left(p^{-1}\left( ( \mathfrak{m}_{\mathfrak{p}}/\mathfrak{h})  _{+}\right) \oplus \mathfrak{n}_{\mathfrak{p}}\right)\cap\mathfrak{g}/\mathfrak{h}=0.
\]
Hence $\left(\mathfrak{m}_{\mathfrak{p}}/\mathfrak{h}\right)_{+}\oplus p(\mathfrak{n}_{\mathfrak{p}})$ defines a complex structure on $(\mathfrak{g}/\mathfrak{h})_\mathbb{C}$.
Also since $J_1$ is integrable $p^{-1}\left((\mathfrak{m}_{\mathfrak{p}}/\mathfrak{h})_{+}\right)$ is a Lie subalgebra of $(\mathfrak{m}_\mathfrak{p})_{\mathbb{C}}$, which also normalizes $\mathfrak{n}_{\mathfrak{p}}$.
Thus  $p^{-1}\left((\mathfrak{m}_{\mathfrak{p}}/\mathfrak{h})_{+}\right) \oplus \mathfrak{n}_{\mathfrak{p}}$ is a Lie algebra. By  Lemma \ref{lmm:12},  its induced complex structure $J$ is integrable.
\end{proof}

\begin{theorem}
\label{thm:22} \label{Param2}If $\mathfrak{g}\supset\mathfrak{h}$ satisfy the
conditions of Theorem \ref{Wang} then the set of integrable $\mathfrak{h}
$--invariant complex structures on $\mathfrak{g}/\mathfrak{h}$ is the set
$\{P(\mathfrak{p},J_{1})\, |\, \mathfrak{p}\in P(\mathfrak{h}),J_{1}$ a
complex structure on $\mathfrak{m}_{\mathfrak{p}}/\mathfrak{h}\}$.
\end{theorem}

\begin{proof}
This is a direct consequence of Corollary \ref{BasicObs}.
\end{proof}
\begin{theorem}
\label{LieAlgIrr}
Let $\mathfrak{g}$ and $\mathfrak{h}$ be as in the previous theorem. If
$\mathfrak{g}/\mathfrak{h}(\ne 0)$ has an $\mathfrak{h}$--invariant, integrable
complex structure and $\mathfrak{h}$ acts effectively and irreducibly on $%
\mathfrak{g}/\mathfrak{h}$ then $\mathfrak{g}$ is semisimple and the pair $(%
\mathfrak{g},\mathfrak{h})$ is an irreducible Hermitian symmetric pair of
compact type.
\end{theorem}

\begin{proof}
By the above we may assume that the  complex structure $J$ is $J(%
\mathfrak{\mathfrak{p}},J_{1})$ with $\mathfrak{p}$ a parabolic subalgebra
of $\mathfrak{g}_{\mathbb{C}}$ such that $\mathfrak{m}=\mathfrak{p\cap
\mathfrak{g}}$ contains $\mathfrak{h}$ and $[\mathfrak{m},\mathfrak{m}]=[\mathfrak{h},\mathfrak{h}]$
and $J_{1}$ is a complex structure on $\mathfrak{m}/\mathfrak{h}$.
 Since we are assuming that the action of $\mathfrak{h}$ is irreducible on $\mathfrak{g}/\mathfrak{h}$  we must have $\mathfrak{m}=\mathfrak{h}$ or $\mathfrak{m}=\mathfrak{g}.$
If $\mathfrak{m}=\mathfrak{g}$ then $\mathfrak{h}$ would act by $0$ on
$\mathfrak{g}/\mathfrak{h}$ but $\dim \mathfrak{g}/\mathfrak{h}\geq 2$ which
contradicts the irreducibility. Using $J$ we  consider the action of $%
\mathfrak{h}$ on $\mathfrak{g}/\mathfrak{h}$ as a complex representation it
is equivalent to the representation on $\left( \mathfrak{g}/\mathfrak{h}%
\right) _{+}$.  Thus, Corollary \ref{BasicObs}  implies that (in the notation of the
corollary) that $\mathfrak{\mathfrak{n}}$ is equivalent with $\mathfrak{g}/%
\mathfrak{h}$ as a real representation. Since $\left[ \mathfrak{n},\mathfrak{%
n}\right] \neq \mathfrak{n}$ and $\left[ \mathfrak{n},\mathfrak{n}\right] $
is $\mathfrak{h}$--invariant, $\mathfrak{n}$ is abelian. Also if $\mathfrak{%
\bar{n}}$ is the complex conjugate of $\mathfrak{n}$, since $\mathfrak{h}=\mathfrak{m}$, then
\[
\mathfrak{g}_{\mathbb{C}}=\mathfrak{h}_{\mathbb{C}}\oplus \left(\bar{\mathfrak{n}}%
\oplus \mathfrak{n}\right).
\]%
Now, $[\mathfrak{\bar{n}\oplus
\mathfrak{n}},\mathfrak{\bar{n}\oplus \mathfrak{n}}]=[\mathfrak{\bar{n},%
\mathfrak{n].}}$ We claim that
\begin{equation}\label{eq:helpsym}
\lbrack \mathfrak{\bar{n},\mathfrak{n}}]\subset \mathfrak{h}_{\mathbb{C}}%
\mathfrak{.}
\end{equation}
Indeed, since $\mathfrak{h}=\mathfrak{m}$
\[
\mathfrak{p=\mathfrak{h}}_{\mathbb{C}}\oplus \mathfrak{n.}
\]%
Also, $\mathfrak{h=\mathfrak{c}}_{\mathfrak{h}}\oplus \lbrack \mathfrak{h},%
\mathfrak{h]}$ and since we have shown that, since $\mathfrak{h}=\mathfrak{m}$,  $\mathfrak{h}$ is the centralizer
of $\mathfrak{c}_{\mathfrak{h}}$ in $\mathfrak{g,}$ and our assumptions
imply $\mathfrak{h\neq \mathfrak{g.}}$ Thus $\mathfrak{\mathfrak{c}}_{%
\mathfrak{h}}\neq 0.$ If $\alpha $ is a non-zero weight of $\left( \mathfrak{%
c}_{\mathfrak{h}}\right) _{\mathbb{C}}$ on $\mathfrak{n}$ then the weight
space for $\alpha $ in $\mathfrak{n}$ is $\mathfrak{h}$--invariant and
non-zero. This implies, by irreducibility, that $\left( \mathfrak{c}_{%
\mathfrak{h}}\right) _{\mathbb{C}}$ acts on $\mathfrak{n}$ by multiplication
by $\alpha $. Noting that $\alpha (\mathfrak{c}_{\mathfrak{h}})\subset \mathbf{i}%
\mathbb{R}$ we see that $\left( \mathfrak{c}_{\mathfrak{h}}\right) _{\mathbb{%
C}}$ acts on $\mathfrak{\bar{n}}$ by $-\alpha.$ Thus $[\mathfrak{\bar{n},%
\mathfrak{n}}]$ is contained in the centralizer of $\mathfrak{c}_{\mathfrak{h%
}}$ in $\mathfrak{g}_{\mathbb{C}}$ so $[\mathfrak{\bar{n},\mathfrak{n}}%
]\subset \mathfrak{h}_{\mathbb{C}}$, as asserted.

If we set $V=\mathfrak{(\bar{n}\oplus \mathfrak{n)\cap \mathfrak{g}}}$ then by (\ref{eq:helpsym}), $%
\mathfrak{g}=\mathfrak{h}\oplus V$ with $[V,V]\subset \mathfrak{h}$. Define $%
\theta :\mathfrak{g}\rightarrow \mathfrak{g}$ to be $\operatorname{id}$ on $\mathfrak{h}$
and $-\operatorname{id}$ on $V$. Then $\theta $ defines an involutive automorphism of $%
\mathfrak{g}$ hence $\theta :\mathfrak{c}_{\mathfrak{g}}\rightarrow
\mathfrak{c}_{\mathfrak{g}}$. This implies that%
\[
\mathfrak{c}_{\mathfrak{g}}=\mathfrak{h}\cap \mathfrak{c}_{\mathfrak{g}%
}\oplus V\cap \mathfrak{c}_{\mathfrak{g}}\text{.}
\]%
Now $\mathfrak{h}\cap \mathfrak{c}_{\mathfrak{g}}=0$ since the action is
effective. If $V\cap \mathfrak{c}_{\mathfrak{g}}\neq 0$ then the
irreducibility would imply that $\dim \mathfrak{g}/\mathfrak{h}=1$ which is
impossible since $\mathfrak{g}/\mathfrak{h}$ is even dimensional. Thus $%
\mathfrak{g}$ is semisimple, and by the above  Cartan decomposition we know that $(\mathfrak{g}, \mathfrak{h})$
is a Hermitian symmetric pair of compact type.
\end{proof}

\section{Applications}

In this section we show how results of Wang and Tits follow from
the Lie algebraic results in this paper. We first note that if $G$ is a
compact Lie group and $H$ is a closed subgroup such that $G/H$ has a
$G$--invariant complex structure. In other words, $G/H$ has an almost complex
structure $J$ that is $G$--invariant and the corresponding Nijenhuis tensor is
$0$.  Let $\mathfrak{g}$ and $\mathfrak{h}$ be respectively the Lie algebras of
$G$ and $H$ with $\mathfrak{h}\subset\mathfrak{g}$ and $J$ is a $G$--invariant
almost complex structure on $G/H$ then under the natural identification of
$\mathfrak{g}/\mathfrak{h}$ with $T_{eH}(G/H)$  we can pull back $J_{eH}$ to
$\mathfrak{g}/\mathfrak{h}$ to make it into a complex vector space with
$ad(\mathfrak{h)}$ acting by complex transformations. If $N $ is the Nijenhuis
tensor corresponding to $J$ then the pull back of $N_{eH}$ to $\mathfrak{g}%
/\mathfrak{h}$ is our tensor $N$. Thus our integrability condition implies
that if $G$ and $H$ are connected then $G/H$ has a $G$--invariant complex
structure with $J$ the associated almost complex structure.

The first result is an extension of results of Wang \cite{Wang} and Tits
\cite{Tits}

\begin{theorem}
\label{thm:31} Let $G$ be a compact connected Lie group and $H$ a closed
connected subgroup. Then $G/H$ has a $G$--invariant complex structure if and
only if $H $ is a subgroup of another subgroup $M$ which is the centralizer of
a torus in $G$ such that $[M,M]=[H,H]$.
\end{theorem}

\begin{proof}
In light of the remarks above this result is a direct consequence of Theorem
\ref{Wang}.
\end{proof}

The next result is a geometric version of our classification theorem.

\begin{theorem}
\label{thm:32} Let $G$ be a compact connected Lie group, $H$ a closed,
connected subgroup and $M$ the centralizer of a torus in $G$ such that
$[M,M]=[H,H]$.  In light of Lemma \ref{lemma:HK} we may also assume that $M \supset H$.
 Then the $G$--invariant complex structures on $G/M$ are
parametrized by the finite number of parabolic subalgebras $\mathfrak{p}%
\subset\mathfrak{g}_{\mathbb{C}}$ such that $\mathfrak{p}\cap
\mathfrak{g=\mathfrak{m}}$. The complex structures on $G/H$ are parametrized
by the pairs of $(\mathfrak{p},J_{1})$ of a parabolic subalgebra
$\mathfrak{p}\subset\mathfrak{g}_{\mathbb{C}}$ such that $\mathfrak{p}%
\cap\mathfrak{g=\mathfrak{m}}$ and an $M$--invariant complex structure,
$J_{1}$, on the torus $M/H$.
\end{theorem}

\begin{proof}
By the remarks above the theorem is a direct consequence of Theorem
\ref{Param2}.
\end{proof}

We will now give proofs of the theorems stated at the end of the introduction.

\begin{proof}{\it of Theorem \ref{gen-Wang}}
We use the notation and assumptions of Theorem \ref{gen-Wang}. Let $\mathfrak{h}$ and $%
\mathfrak{g}$ be respectively the Lie algebras of $H$ and $G$. Let $J$ be
the integrable complex structure on $\mathfrak{g}/\mathfrak{h}$ and let $%
\mathfrak{m}\supset \mathfrak{h}$ be as in Theorem \ref{thm:The-parabolic} for $J$. Then
Corollary \ref{coro:12} implies that the connected subgroup corresponding to $\mathfrak{%
m}$ in $G$, $M$ is given as in the statement. Since $M$ is the centralizer
of a torus in $G$ and $G$ is connected, $G/M$ is a flag variety. The
fibration is now obvious.
\end{proof}
Next we give the proof of Theorem \ref{description}.

\begin{proof}{\it of Theorem \ref{description}}
Let $M$ be as in Theorem \ref{description} and let $T$ be the identity component of the
center of $M$. Then $M=T[M,M]$ so $M=TH.$ Let $T_{1}$ be a maximal torus in $%
C_{G}(H)$ containing $T$. Since $M$ is the centralizer of its center this
implies that $T_{1}\subset M$. Since $M=TH$ we see that $T_{1}$ is in the
center of $M$ thus $T_{1}=T$. Finally, Theorem \ref{thm:22} implies the description of
the $G$--invariant complex structures on $G/H$.
\end{proof}

Next we give a proof of Theorem \ref{Herm} in the introduction.

\begin{proof}{\it of Theorem \ref{Herm}}
Let $\left\langle \cdot, \cdot\right\rangle $ be a homogeneous, positive
definite Hermitian structure on $M$ and let $G$ be the identity component of
the isometry group of $\left\langle \cdot, \cdot\right\rangle $. Since $M$ is
connected if $(M,\left\langle \cdot, \cdot\right\rangle )$ is homogeneous then
$G$ acts transitively. Since $M$ is compact, $G$ is compact. Conversely if $G$
is a compact Lie group and $H\subset G$ is a closed subgroup and $J$ is a $G$
invariant almost complex structure on $G$ then since $H$ is compact, it is
obvious that $G/H$ has a $G$--invariant, positive definite  Hermitian structure.
\end{proof}

Finally we note that Theorem \ref{Irreducible} follows from Theorem \ref{LieAlgIrr} and Theorem \ref{Herm}.

\acknowledgements{\rm We  thank McKenzie Wang for helpful comments, Fabio Podest\'a for sending in an alternate argument for Theorem \ref{LieAlgIrr} and his comments. We also thank the referees for their time and comments, Yuguang Shi for the invitation of this contribution.}

\end{document}